\documentclass{article}
\usepackage{arxiv}
\usepackage[toc,page]{appendix}

\usepackage[usenames]{xcolor}

\usepackage{inputenc} %
\usepackage[T1]{fontenc}    %
\usepackage{hyperref}       %
\usepackage{url}            %
\usepackage{booktabs}       %
\usepackage{amsfonts}       %
\usepackage{nicefrac}       %
\usepackage{microtype}      %
\usepackage{lipsum}
\usepackage{mathtools}
\usepackage{cite}
\usepackage{amsmath}
\usepackage{amssymb}
\usepackage{amsthm}
\usepackage{algorithm,algcompatible}
\usepackage[toc,page]{appendix}

\usepackage{lineno}
\mathtoolsset{showonlyrefs}

\newtheorem{theorem}{Theorem}
\newtheorem{lemma}{Lemma}

\newtheorem{proposition}{Proposition}

\numberwithin{equation}{section}

\DeclareMathOperator*{\esssup}{ess\,sup}

\usepackage{mathptmx}      %
\begin{document}

\title{Sharp Lower Bounds on the Manifold Widths of Sobolev and Besov Spaces
}

\author{Jonathan W. Siegel \\
 Department of Mathematics\\
 Texas A\&M University\\
 College Station, TX 77843 \\
 \texttt{jwsiegel@tamu.edu}
}

\maketitle

\begin{abstract}
    We consider the problem of determining the manifold $n$-widths of Sobolev and Besov spaces with error measured in the $L_p$-norm. The manifold widths control how efficiently these spaces can be approximated by general non-linear parametric methods with the restriction that the parameter selection and parameterization maps must be continuous. Existing upper and lower bounds only match when the Sobolev or Besov smoothness index $q$ satisfies $q\leq p$ or $1 \leq p \leq 2$. We close this gap and obtain sharp lower bounds for all $1 \leq p,q \leq \infty$ for which a compact embedding holds. A key part of our analysis is to determine the exact value of the manifold widths of finite dimensional $\ell^M_q$-balls in the $\ell_p$-norm when $p\leq q$. Although this result is not new, we provide a new proof and apply it to lower bounding the manifold widths of Sobolev and Besov spaces. Our results show that the Bernstein widths, which are typically used to lower bound the manifold widths, decay asymptotically faster than the manifold widths in many cases. 
\end{abstract}

\section{Introduction}
Due in part to the practical success of deep neural networks \cite{lecun2015deep}, non-linear methods of approximation have gained in importance in recent years. In this work, we consider limitations on general non-linear methods of approximation which use a finite number of parameters. Specifically, given a Banach space $X$ and a compact set $K\subset X$, we consider the manifold $n$-widths introduced in \cite{devore1989optimal}, defined by
\begin{equation}\label{manifold-widths-definition}
    \delta_n(K)_X := \inf_{a_n,M_n}\sup_{f\in K}\|f - M_n(a_n(f))\|_X,
\end{equation}
where the infimum is taken over all pairs of \textit{continuous} maps $a_n:K\rightarrow \mathbb{R}^n$ and $M_n:\mathbb{R}^n\rightarrow X$. We can think of the map $a_n$ as being an encoding map and the map $M_n$ as being a decoding map.

As remarked in \cite{devore1989optimal}, if the continuity assumption is dropped, this notion becomes vacuous since using a space-filling curve any set $K$ in a separable Banach space $X$ can be `parameterized' by a single real number. The manifold widths $\delta_n$ control the best possible rates of approximation using general parametric methods with $n$ parameters, if we additionally require that both the parameter selection and parameterization maps are continuous.

The problem we study in this work is the determination of the manifold widths $\delta_n(K)_X$ where $K$ is the unit ball of a Besov or Sobolev space and the error is measured in $L_p$. Let us begin by recalling the definitions of Besov and Sobolev spaces. 

Let $\Omega\subset \mathbb{R}^d$ be a bounded domain, which we take to be the unit cube $\Omega = [0,1]^d$ for simplicity. We remark that our results can be transferred to more general domains, with appropriately modified constants, in a standard manner. 

We write
$L_p(\Omega)$ for the set of functions for which the $L_p$-norm on $\Omega$, defined by
$$
    	\|f\|_{L_p(\Omega)} = \left(\int_{\Omega}|f(x)|^pdx\right)^{1/p} < \infty,
$$
is finite. When $p=\infty$, the $L_\infty$-norm is defined by $\|f\|_{L_\infty(\Omega)} = \esssup_{x\in \Omega} |f(x)|$.

Given an integer $s \geq 1$ and $1\leq q\leq \infty$, we define the Sobolev space $W^s(L_q(\Omega))$ to be the set of $f\in L_q(\Omega)$ which have weak derivatives of order $s$ in $L_q(\Omega)$ with norm given by
\begin{equation}
    \|f\|_{W^s(L_q(\Omega))} := \|f\|_{L_q(\Omega)} + \|f^{(s)}\|_{L_q(\Omega)},
\end{equation}
where $f^{(s)}$ denotes the (tensor of) weak derivatives of order $s$ (see \cite{evans2022partial}, Chapter 5, for details).

Besov spaces, which are conveniently defined in terms of moduli of smoothness, provide a well-known generalization of Sobolev spaces to non-integer smoothness $s$. For a function $f\in L_q(\Omega)$ and an integer $k \geq 1$, we define its $k$-th order modulus of smoothness by
\begin{equation}\label{definition-of-modulus-of-smoothness}
    \omega_k(f,t)_q = \sup_{|h|\leq t}\|\Delta^k_h f\|_{L_q(\Omega_{kh})},
\end{equation}
where the supremum is over $h\in \mathbb{R}^d$ and the $k$-th order finite differences $\Delta^k_h$ are defined by
$$
    \Delta^k_h f(x) = \sum_{j=0}^k (-1)^j\binom{k}{j} f(x+jh).
$$
Here the $L_q$ norm is taken over the set $\Omega_{kh} := \{x\in \Omega,~x + kh\in \Omega\}$ to guarantee that all terms of the relevant finite differences are contained in $\Omega$. We remark that when $k = 1$ and $q=\infty$, the modulus of smoothness reduces to the well-known modulus of continuity.

Given parameters $s > 0$, $1\leq r,q\leq \infty$, we define the Besov norm $B^s_r(L_q(\Omega))$ by
\begin{equation}
    \|f\|_{B^s_r(L_q(\Omega))} := \|f\|_{L_q(\Omega)} + |f|_{B^s_r(L_q(\Omega))},
\end{equation}
with Besov semi-norm given by
\begin{equation}\label{besov-semi-norm-definition}
    |f|_{B^s_r(L_q(\Omega))} := \left(\int_0^\infty \frac{\omega_k(f,t)_q^r}{t^{sr+1}}dt\right)^{1/r}
\end{equation}
when $r < \infty$, and by
\begin{equation}
    |f|_{B^s_\infty(L_q(\Omega))} := \sup_{t > 0} t^{-s}\omega_k(f,t)_q,
\end{equation}
when $r = \infty$. Here the index $k$ of the modulus of smoothness must satisfy $k > s$, and all such choices of $k$ give equivalent norms. We denote by $B^s_r(L_q(\Omega))$ the Besov space of functions $f\in L_q(\Omega)$ whose $B^s_r(L_q(\Omega))$-norms are finite. Besov spaces play a central role in approximation theory, signal processing, and applied mathematics, and the space $B^s_r(L_q(\Omega))$ can be thought of as a space of functions with $s$ derivatives in $L_q$ (with the index $r$ providing a finer gradation of these spaces). For more detailed information on Besov spaces, we refer to \cite{devore1993constructive,devore1984maximal,devore1993besov,triebel2008function}. We also remark that the commonly used fractional Sobolev spaces \cite{di2012hitchhikerʼs} correspond to Besov spaces with $q = r$. 

Consider the classes 
\begin{equation}
    \mathcal{F}_{q}^s := \{f\in W^s(L_q(\Omega)),~\|f\|_{W^s(L_q(\Omega))} \leq 1\}~\text{and}~\mathcal{B}_{r,q}^s := \{f\in B^s_r(L_q(\Omega)),~\|f\|_{B^s_r(L_q(\Omega))} \leq 1\},
\end{equation}
which are the unit balls of the corresponding Sobolev and Besov spaces, respectively.
We wish to determine (asymptotically) the manifold widths $\delta_n(\mathcal{F}_{q}^s)_{L_p(\Omega)}$ and $\delta_n(\mathcal{B}_{r,q}^s)_{L_p(\Omega)}$ for different values of $s,r,q,p$. 

In order for this problem to make sense, we need $\mathcal{F}_{q}^s$ or $\mathcal{B}_{r,q}^s$ to be a compact subset of $L_p(\Omega)$. It is well-known that this is equivalent to the Sobolev embedding condition
\begin{equation}\label{sobolev-embedding-condition}
    \frac{1}{q} - \frac{1}{p} < \frac{s}{d}
\end{equation}
for bounded domains $\Omega$ (see \cite{evans2022partial}, Chapter 5, for instance). 

To put this problem into perspective, let us discuss other classical notions of width and how they relate to the manifold widths. We will only scratch the surface of this subject here and refer to \cite{lorentz1996constructive}, Chapters 13 and 14, \cite{pinkus2012n}, and \cite{tikhomirov2012analysis}, Chapter 3 for a more detailed presentation of this material.

 Essentially, by putting different restrictions on the encoding map $a_n$ and the decoding maps $M_n$, we obtain different widths, which measure the complexity of the set $K$ in different ways. For example, consider the Kolmogorov widths \cite{kolmogoroff1936uber}, which are typically defined as
\begin{equation}
    d_n(K)_X := \inf_{V_n} d_X(K,V_n),
\end{equation}
where the infimum is taken over all linear subspaces $V_n\subset X$ of dimension at most $n$, and $$d_X(K,V_n) = \sup_{f\in K}\inf_{f_n\in V_n} \|f - f_n\|_X$$ denotes the distance of $K$ to $V_n$. The Kolmogorov widths correspond to restricting the decoding map $M_n$ to be linear with no restrictions on the encoding map $a_n$.

The Gelfand widths are defined by
\begin{equation}
    d^n(K)_X := \inf_{V_n}\sup\{\|f\|_X,~f\in K\cap V_n\},
\end{equation}
where the infimum is taken over all co-dimension $n$ subspaces $V_n$. For convex centrally symmetric sets $K$, these
correspond to restricting the encoding map $a_n$ to be linear with no restriction on the decoding map $M_n$. 

Finally, the linear widths, defined by
\begin{equation}
    d^L_n(K)_X := \inf_{T_n}\sup_{f\in K}\|f - Tf\|_X,
\end{equation}
where the infimum is taken over all rank $n$ operators $T:X\rightarrow X$,
correspond to restricting both $a_n$ and $M_n$ to be linear. We remark that there are numerous other non-linear widths which have recently been introduced and which fit into this framework, such as the stable manifold widths \cite{cohen2022optimal} and the Lipschitz widths \cite{petrova2023lipschitz}. These new notions of non-linear widths enforce conditions stronger than just continuity which are arguably more relevant to practical applications.

Another important and closely related notion of width, which cannot be put into the same framework based upon restricting the encoding and decoding maps $a_n$ and $M_n$, are the Bernstein widths \cite{tikhomirov1971some}, defined by
\begin{equation}
    b_n(K)_X = \sup_{V_n}\inf_{f\in V_n\cap K^c} \|f\|_X.
\end{equation}
Here the supremum is over all subspaces $V_n\subset X$ of dimension $n+1$. The Bernstein widths measure the size of the largest ball completely contained in an $(n+1)$-dimensional section of $K$. The importance of the Bernstein widths stems partially from the fact that they provide a lower bound on the Gelfand widths, linear widths, Kolmogorov widths and manifold widths via the Borsuk-Ulam theorem \cite{borsuk1933drei} (see \cite{devore1989optimal} and \cite{lorentz1996constructive}, Chapter 13).

Finally, let us mention that the manifold widths are closely related to the Aleksandrov widths defined via $n$-dimensional cell complexes (see \cite{tikhomirov2012analysis}, Chapter 3 and \cite{devore1993wavelet,stesin1975aleksandrov}). Indeed, as a consequence of the Pontryagin-N\"obeling principle (see \cite{aleksandrov1998combinatorial}, Chapter 4, Theorem 1.9) it can be shown that the manifold widths and Aleksandrov widths are equivalent up to possibly shifting the index $n$ by a factor of $2$ (see Lemma 2.2 in \cite{devore1993wavelet}).

In the following, we consider widths of the compact sets $\mathcal{F}_{q}^s$ and $\mathcal{B}_{r,q}^s$, consisting of the unit balls of Sobolev and Besov spaces which compactly embed in $L_p(\Omega)$. 

The asymptotic rate of decay of the Kolmogorov, Gelfand, linear, and Bernstein widths of these sets are known. These are results of a rather deep theory developed during the past century. We refer to \cite{lorentz1996constructive}, Chapters 13 and 14, and \cite{pinkus2012n}, Chapter 7, and the references contained in both of these books for the theory giving the Kolmogorov, Gelfand, and linear $n$-widths of Sobolev and Besov spaces. The determination of the Bernstein widths can be found in \cite{tikhomirov2012analysis}, Chapter 3 and the references therein (see also \cite{li2013bernstein,feng2009bernstein}). Since we will use them later, let us recall that the Bernstein widths are given by
\begin{equation}\label{bernstein-widths-asymptotics}
    b_n(\mathcal{F}_{q}^s)_{L_p(\Omega)} \eqsim b_n(\mathcal{B}_{r,q}^s)_{L_p(\Omega)} \eqsim \begin{cases}
        n^{-s/d} & p \geq q~\text{or}~1 \leq p \leq q \leq 2\\
        n^{-s/d + 1/q - 1/2} & 1 \leq p \leq 2 \leq q\leq \infty\\
        n^{-s/d + 1/q - 1/p} & 2 \leq p \leq q \leq \infty,
    \end{cases}
\end{equation}
whenever the compact embedding condition \eqref{sobolev-embedding-condition} holds.

In contrast, the manifold widths $\delta_n$ have not been determined for the Sobolev and Besov balls $\mathcal{F}_{q}^s$ and $\mathcal{B}_{r,q}^s$ with respect to the $L_p$-norm in all cases for which the compact embedding \eqref{sobolev-embedding-condition} holds.

This problem, which quantifies the limits of continuous non-linear approximation for classical smoothness spaces, was first considered in \cite{stesin1975aleksandrov} for the closely related Aleksandrov widths in place of the manifold widths. There the Sobolev unit ball $\mathcal{F}_{q}^s$ for integer $s \geq 1$ in one-dimension is considered, and it is shown that Aleksandrov widths decay like $n^{-s}$ (with upper and lower bounds matching up to a constant). 

The case of general Besov spaces (and manifold widths instead of Aleksandrov widths) was first considered in \cite{devore1989optimal}. There it is shown using spline approximation that in one dimension we have the upper bound
\begin{equation}
\delta_n(\mathcal{B}_{r,q}^s)_{L_p(\Omega)} \leq Cn^{-s}
\end{equation}
for integral $s$ and all $s,r,q$ and $p$ for which \eqref{sobolev-embedding-condition} holds (and the analogous result for the Sobolev ball $\mathcal{F}_{q}^s$). 
In \cite{devore1993wavelet} this upper bound was extended to all dimensions $d$ and smoothness $s > 0$, i.e. it was shown that
\begin{equation}\label{manifold-upper-bound} \delta_n(\mathcal{B}_{r,q}^s)_{L_p(\Omega)} \leq Cn^{-s/d}
\end{equation}
as long as \eqref{sobolev-embedding-condition} holds (here the constants $C$ are independent of $n$ but depend upon the parameters $s,r,p,q$).

 Concerning lower bounds, in \cite{devore1989optimal} it is shown that if $q\leq p$ or if $1 \leq p \leq q\leq 2$, then the upper bound in \eqref{manifold-upper-bound} is sharp, i.e. that in this case we have the bound
\begin{equation}\label{devore-lower-bounds}
    \delta_n(\mathcal{B}_{r,q}^s)_{L_p(\Omega)} \geq Cn^{-s/d}
\end{equation}
for another (potentially different) constant $C$. The tool used to prove lower bounds on the manifold widths are the Bernstein widths. In particular, utilizing the Borsuk-Ulam theorem \cite{borsuk1933drei} one can show that (see \cite{devore1989optimal}, Theorem 3.1)
\begin{equation}\label{bernstein-manifold-lower-bound}
    b_n(K)_X \leq \delta_n(K)_X
\end{equation}
for any compact set $K\subset X$. Applying the known asymptotics of the Bernstein widths \eqref{bernstein-widths-asymptotics}, we obtain \eqref{devore-lower-bounds} when either $q \leq p$ or $1 \leq p \leq q \leq 2$.

However, in the case $p < q$ and $q > 2$ the Bernstein widths decay faster than the upper bound \eqref{manifold-upper-bound} and this method fails to give a sharp result. We remark that the metric entropy also cannot be used to give lower bounds since, as noted in \cite{cohen2022optimal}, the manifold widths do not obey a Carl's inequality which would relate them to the entropy. The goal of this work is to extend the lower bound \eqref{devore-lower-bounds} to this case and thus to show that the bound \eqref{manifold-upper-bound} is always sharp. This completes the determination of the asymptotic decay of the manifold widths of Sobolev and Besov spaces in all cases. Since the Bernstein widths and metric entropy fail to give a sharp lower bound, this requires new tools.

In order to bridge this gap, we introduce a new notion of width, which we call the \textit{sphere embedding widths}, defined by
\begin{equation}\label{sphere-embedding-widths-def}
    s_n(K)_X := \sup_{c:S^n\rightarrow K}\inf_{z\in S^n} \|c(z)\|_X.
\end{equation}
Here $S^n$ denotes the $n$-dimensional unit sphere
$$
    S^n = \{x\in \mathbb{R}^{n+1},~\|x\| = 1\},
$$
and the supremum above is over all continuous and odd (i.e. $c(-x) = -c(x)$) functions $c:S^n\rightarrow K$. The sphere embedding widths more precisely capture the obstruction imposed by the Borsuk-Ulam theorem and in Proposition \ref{sphere-manifold-lower-bound-prop} in Section \ref{manifold-widths-lower-bound-section} we show that
\begin{equation}\label{sphere-embedding-general-bound}
    s_n(K)_X \leq \delta_n(K)_X,
\end{equation}
analogous to \eqref{bernstein-manifold-lower-bound}. However, the sphere embedding widths can decay slower than the Bernstein widths. In particular, for the classes $\mathcal{B}_{r,q}^s\subset L_p(\Omega)$ we show in the proof of Theorem \ref{main-lower-bound-theorem} that
\begin{equation}
    s_n(\mathcal{B}_{r,q}^s)_{L_p(\Omega)} \geq Cn^{-s/d}
\end{equation}
for all $s > 0$, $1\leq r,p,q\leq \infty$ for which the Sobolev embedding condition \eqref{sobolev-embedding-condition} holds. This extends the lower bound \eqref{devore-lower-bounds} to all $1 \leq p, q \leq \infty$.

The key ingredient in the proof is a lower bound on the sphere embedding widths of unit $\ell_q^M$-balls with respect to the $\ell_p$-norm. Indeed, writing 
\begin{equation}\label{ell-q-unit-ball-definition}
    K_q^M := \left\{x\in \mathbb{R}^M,~\sum_{i=1}^M |x_i|^q \leq 1\right\}
\end{equation}
for the unit ball in $\ell_q^M$, we prove in Proposition \ref{sphere-embedding-lq-lower-bounds-prop} that
\begin{equation}
    s_n(K_q^M)_{\ell_p} \geq (M - n)^{1/p - 1/q}
\end{equation}
for $n < M$. This enables us to determine the precise value of the manifold widths of unit $\ell_q^M$-balls with respect to the $\ell_p$-norm in the regime $p \leq q$. In particular, we show in Theorem \ref{manifold-widths-lq-lp-theorem} that
\begin{equation}
    \delta_n(K^M_q)_p = (M - n)^{1/p - 1/q}
\end{equation}
whenever $p \leq q$.
This result was first proved in \cite{khodulev1989remark}, and although it is not new, we provide a new proof using the sphere embedding widths. We remark that combined with the results in \cite{devore1989optimal,devore1993wavelet} and \cite{tikhomirov2012analysis}, Chapter 3, which determine the manifold widths (and Aleksandrov widths) of $\ell_q^M$-balls with respect to the $\ell_p$-norm in the regime $p \geq q$, this gives the manifold widths of finite dimensional $\ell_q$-balls.

\section{Sharp Bounds on the Manifold Widths of $\ell_q^M$-balls}\label{manifold-widths-lp-balls}
In this Section, we precisely determine the manifold widths of the $\ell_q^M$ unit ball with respect to the $\ell_p$-norm in the regime $p \leq q$. These are given by the following Theorem, first proved in \cite{khodulev1989remark}.
\begin{theorem}[Main result in \cite{khodulev1989remark}]\label{manifold-widths-lq-lp-theorem}
    Let $1\leq n < M$ be integers, $0< p \leq q\leq \infty$ and let $K^M_q$ denote the $\ell_q^M$ unit ball as in \eqref{ell-q-unit-ball-definition}. Then the manifold widths $\delta_n$ satisfy
    \begin{equation}
        \delta_n(K^M_q)_p = (M - n)^{1/p - 1/q}.
    \end{equation}
\end{theorem}
Although this result is not new, we give a new proof using the sphere embedding widths \eqref{sphere-embedding-widths-def}, which is based on fundamentally different ideas than the argument in \cite{khodulev1989remark}. We believe that our method of proof using the sphere embedding widths may have potential applications to other related problems.

Theorem \ref{manifold-widths-lq-lp-theorem} consists of two parts, an upper bound and a lower bound. The upper bound is simple and is obtained by letting the encoding map $a_n$ take the first $n$-coordinates of $x\in \mathbb{R}^M$, and the decoding map $M_n$ be the inclusion of $\mathbb{R}^n$ into $\mathbb{R}^M$ which sets the last $M-n$ coordinates to $0$. Then for any $x\in \mathbb{R}^M$, we easily see that
\begin{equation}
    x - M_n(a_n(x)) = P_{M-n}(x),
\end{equation}
where $P_{M-n}$ is the projection which sets the first $n$ coordinates of $x$ to $0$. We thus see that
\begin{equation}\label{upper-bound-manifold-equation}
    \delta_n(K^M_q)_p \leq \sup_{x\in \mathbb{R}^M} \frac{\left(\sum_{i=n+1}^M |x_i|^p\right)^{1/p}}{\left(\sum_{i=1}^M |x_i|^q\right)^{1/q}}.
\end{equation}
Since $p \leq q$, H\"older's inequality implies that
\begin{equation}
    \sum_{i=n+1}^M |x_i|^p \leq \left(\sum_{i=n+1}^M |x_i|^q\right)^{p/q}(M-n)^{1-p/q}
\end{equation}
Taking $p$-th roots and plugging this into \eqref{upper-bound-manifold-equation} gives
\begin{equation}\label{manifold-widths-lq-lp-upper-bound-equation}
    \delta_n(K^M_q)_p \leq (M - n)^{1/p - 1/q}.
\end{equation}
Here we must make the obvious modifications if $q = \infty$.
The lower bound in Theorem \ref{manifold-widths-lq-lp-theorem} is more subtle and we prove this using the sphere embedding widths $s_n$ defined in \eqref{sphere-embedding-widths-def}. We first show using the Borsuk-Ulam Theorem that the sphere embedding widths lower bound the manifold widths.
\begin{proposition}\label{sphere-manifold-lower-bound-prop}
    For any compact $K\subset X$ we have the bound
    \begin{equation}
        s_n(K)_X \leq \delta_n(K)_X
    \end{equation}
\end{proposition}
\begin{proof}
    This follows essentially the same argument as the proof of Theorem 3.1 in \cite{devore1989optimal} using the Borsuk-Ulam Theorem \cite{borsuk1933drei}. 
    
    Let $c:S^n\rightarrow K$ be a continuous odd map from the $n$-dimensional sphere $S^n$ into $K$, and $a_n:K\rightarrow \mathbb{R}^n$ an arbitrary continuous map. Then the composition $a_n\circ c$ is a continuous map $S^n\rightarrow \mathbb{R}^n$ so by the Borsuk-Ulam theorem there is a point $z\in S^n$ such that $a_n(c(z)) = a_n(c(-z))$. Let $f^* = c(z)\in K$. Since $c$ is odd we have $-f^* = c(-z)$, and thus $a_n(f^*) = a_n(-f^*)$. Then for any reconstruction map $M_n:\mathbb{R}^n\rightarrow K$, we have
    \begin{equation}
    \begin{split}
        \sup_{f\in K} \|f - M_n(a_n(f))\|_{X}&\geq \max\left\{\|f^* - M_n(a_n(f^*))\|_{X},\|f^* + M_n(a_n(f^*))\|_{X}\right\}\\
        &\geq \|f^*\|_{X} \geq \inf_{z\in S^n} \|c(z)\|_X.
    \end{split}
    \end{equation}
    Since $M_n,a_n$ and $c$ are arbitrary we get $s_n(K)_X \leq \delta_n(K)_X$ as desired.
\end{proof}

Next, we lower bound the sphere embedding widths for the $\ell_q^N$-unit ball with respect to the $\ell_p$-norm. This is the main new technical result of the paper.
\begin{proposition}\label{sphere-embedding-lq-lower-bounds-prop}
    Let $1\leq n < M$ be integers, $0< p \leq q\leq \infty$ and let $K^M_q$ denote the $\ell_q^M$ unit ball as in \eqref{ell-q-unit-ball-definition}. Then the sphere embedding widths satisfy
    \begin{equation}
        s_n(K^M_q)_p \geq (M - n)^{1/p - 1/q}
    \end{equation}
\end{proposition}
Combining this with Proposition \ref{sphere-manifold-lower-bound-prop} and the upper bound \eqref{manifold-widths-lq-lp-upper-bound-equation} proves Theorem \ref{manifold-widths-lq-lp-theorem}.

The proof of Proposition \ref{sphere-embedding-lq-lower-bounds-prop} depends upon the following elementary Lemma.
\begin{lemma}\label{non-zero-coordinate-lemma}
    Let $1\leq N \leq M$ be integers. There exists an $N$-dimensional subspace $V_N\subset \mathbb{R}^{M}$ such that for any $0\neq x\in V_N$ we have
    \begin{equation}
        |\{i:~x_i=0\}| < N.
    \end{equation}
    In other words, a non-zero vector $x\in V_N$ cannot vanish in $N$ coordinates.
\end{lemma}
\begin{proof}
    Indeed, a randomly chosen subspace of $\mathbb{R}^{M}$ has this property with probability $1$. To see this, consider the $M\times N$ matrix $A$ whose columns are a basis for $V_N$. If every $N\times N$ minor of $A$ is non-singular, then $V_N$ has the desired property. Indeed, let $x\in V_N$ which means that
    \begin{equation}\label{x-equation}
        x = \sum_{i=1}^N c_ia^i,
    \end{equation}
    where $a^i$ are the columns of the matrix $A$ and $c_i$ are coefficients. For any set $S\subset \{1,...,M\}$ of $N$ coordinates, we then have
    $$
        c = A_S^{-1}x_S,
    $$
    where $c$ denotes the vector of coefficients in \eqref{x-equation}, $A_S$ denotes the $N\times N$ minor of $A$ corresponding to $S$ (which is non-singular by assumption), and $x_S$ denotes the vector $x$ restricted to the coordinates in $S$. Thus, if $x$ vanishes in the coordinates $S$, then $c = 0$ so $x = 0$.

    For example, taking the span of the columns of a Gaussian random matrix of size $M\times N$ gives such a subspace with probability $1$, since with probability $1$ each $N\times N$ minor of this matrix is non-singular.
\end{proof}
\begin{proof}[Proof of Proposition \ref{sphere-embedding-lq-lower-bounds-prop}]
    We use Lemma \ref{non-zero-coordinate-lemma} to construct a continuous antipodal map $c:S^{n}\rightarrow K_q^M$ such that
    \begin{equation}
        \|c(z)\|_{\ell_p} \geq (M - n)^{1/p - 1/q}
    \end{equation}
    for all $z\in S^n$. Choose a subspace $V_{n+1}\subset \mathbb{R}^M$ of dimension $n+1$ which satisfies the conditions of Lemma \ref{non-zero-coordinate-lemma} with $N = n+1$, i.e. such that any $0\neq x\in V_{n+1}$ can vanish in at most $n$ coordinates. Let $a_1,...,a_{n+1}\in \mathbb{R}^M$ be an orthonormal basis for this space and consider the map $P:S^n\rightarrow \mathbb{R}^M$ given by
    \begin{equation}
        P(z) = \sum_{i=1}^{n+1} z_ia_i,
    \end{equation}
    which simply parameterizes the sphere in the space $V_{n+1}$. Let
    \begin{equation}
        0 \leq \pi_{(1)}(z) \leq \pi_{(2)}(z) \leq \cdots \leq \pi_{(M)}(z)
    \end{equation}
    be a non-decreasing rearrangement of the magnitudes of the coordinates of $P(z)$, and observe that each $\pi_{(i)}$ is a continuous function of $z$. 
    
    By construction, $\pi_{(n+1)}(z) > 0$ for all $z\in S^d$ and since the sphere is compact, we get
    \begin{equation}
        \epsilon := \inf_{z\in S^d} \pi_{(n+1)}(z) > 0.
    \end{equation}
    Define the continuous function
    \begin{equation}
    t_\epsilon(x) = \begin{cases}
        -1 & x\leq -\epsilon\\
        x/\epsilon & -\epsilon\leq x\leq \epsilon\\
        1 & x\geq \epsilon,
    \end{cases}
    \end{equation}
    and define a map $\tilde{P}$ by applying $t_\epsilon$ coordinate-wise to $P(z)$, i.e. we define
    \begin{equation}
        \tilde{P}(z)_i = t_\epsilon(P(z)_i).
    \end{equation}
    Finally, we define the map $c(z)$ by normalizing $\tilde{P}(z)$ in the $\ell_q$-norm, i.e. we set
    \begin{equation}
        c(z) = \frac{\tilde{P}(z)}{\|\tilde{P}(z)\|_{\ell_q}}.
    \end{equation}
    Obviously, by construction we have that $c(z)\in K_q^M$. 
    
    Now let $z\in S^d$.
    From the definition of $\epsilon$, we see that for every $z\in S^n$ at most $n$ coordinates of $P(z)$ have magnitude smaller than $\epsilon$. Hence, upon applying $t_\epsilon$, this implies that at least $M - n$ coordinates of $\tilde{P}(z)$ are $\pm 1$. Letting $x_1,...,x_n$ denote the remaining coordinates, we see that
    \begin{equation}\label{c-lower-bound-equation-452}
        \|c(z)\|_{\ell_p} = \frac{\|\tilde{P}(z)\|_{\ell_p}}{\|\tilde{P}(z)\|_{\ell_q}} = \frac{[(M - n) + \sum_{i=1}^n |x_i|^p]^{1/p}}{[(M - n) + \sum_{i=1}^n |x_i|^q]^{1/q}} \geq \frac{[(M - n) + \sum_{i=1}^n |x_i|^p]^{1/p}}{[(M - n) + \sum_{i=1}^n |x_i|^p]^{1/q}},
    \end{equation}
    since $p\leq q$ and all $|x_i|\leq 1$, so that $|x_i|^q \leq |x_i|^p$. This implies, since $1/p - 1/q \geq 0$, that
    \begin{equation}
        \|c(z)\|_{\ell_p} \geq \left[(M - n) + \sum_{i=1}^n |x_i|^p\right]^{1/p - 1/q} \geq (M-n)^{1/p - 1/q}
    \end{equation}
    for all $z\in S^d$ as desired.
\end{proof}

\section{Lower Bounds on Manifold Widths of Sobolev and Besov Spaces}\label{manifold-widths-lower-bound-section}
In this Section, we utilize Propositions \ref{sphere-manifold-lower-bound-prop} and  \ref{sphere-embedding-lq-lower-bounds-prop} to obtain the following lower bound on the manifold widths of Sobolev and Besov unit balls, which matches the upper bounds proved in \cite{devore1989optimal,devore1993wavelet} in all cases.
\begin{theorem}\label{main-lower-bound-theorem}
    Let $\Omega = [0,1]^d$ be the unit cube and $1\leq p,q,r\leq \infty$. For $n \geq 1$, the manifold widths of the unit ball $\mathcal{B}_{r,q}^s$ of the Besov space $B^s_r(L_q(\Omega))$ satisfy
    \begin{equation}
        \delta_n(\mathcal{B}_{r,q}^s)_{L_p(\Omega)} \geq Cn^{-s/d}
    \end{equation}
    for a constant $C:=C(s,d,p,q)$ independent of $n$.
\end{theorem}
We remark that although we only consider Besov spaces, a completely analogous (and even simpler) proof applies to the unit balls of the Sobolev spaces as well.
\begin{proof}
    It suffices to consider the case $p\leq q$ and $r = 1$. We will use Proposition \ref{sphere-embedding-lq-lower-bounds-prop} to prove that \begin{equation}
        s_n(\mathcal{B}_{1,q}^s)_{L_p(\Omega)} \geq Cn^{-s/d}.
    \end{equation}
    Utilizing Proposition \ref{sphere-manifold-lower-bound-prop} then implies the result.
    
    Fix a non-zero $C^\infty$ function $\psi$ supported on the unit cube $[0,1]^d$. Divide the cube into $M = m^d$ sub-cubes $\Omega_i$ where $$2n \leq M \leq Cn.$$
    Let $x_i$ for $i=1,...,M$ denote the bottom corner of the $i$-th sub-cube and define
    \begin{equation}
        \psi_i(x) = \psi(m(x-x_i)).
    \end{equation}
    Consider the map $\Psi:\mathbb{R}^M\rightarrow L_p$ given by
    \begin{equation}
        \Psi(a) = \sum_{i=1}^M a_i\psi_i.
    \end{equation}
    We note that a change of variables implies the following identity relating the moduli of smoothness of $\psi_i$ and $\psi$,
    \begin{equation}
        \omega_k(\psi_i,t/m)_q = m^{-d/q}\omega_k(\psi,t)_q.
    \end{equation}
    Utilizing the sub-additivity of the modulus of smoothness (see \cite{devore1993besov}), we get
    \begin{equation}
        \omega_k(\Psi(a),t/m)_q^q \leq C\sum_{i=1}^M a_i^q\omega_k(\psi_i,t/m)^q_q = Cm^{-d}\omega_k(\psi,t)^q_q\sum_{i=1}^M a_i^q
    \end{equation}
    where $C$ is a constant depending upon $k$ and $q$. Taking $q$-th roots and integrating with respect to $t$ using the definition \eqref{besov-semi-norm-definition}, we obtain the bound
    \begin{equation}\label{bespv-space-upper-bound}
        \|\Psi(a)\|_{B^s_1(L_q(\Omega))} \leq Cm^{-d/q+s}\|\psi\|_{B^s_1(L_q(\Omega))}\|a\|_{\ell_q} \leq Cm^{-d/q+s}\|a\|_{\ell_q}.
    \end{equation}
    We also easily calculate that
    \begin{equation}\label{l-p-sum-lower-bound}
        \|\Psi(a)\|_{L_p(\Omega)} = m^{-d/p}\|\psi\|_{L_p(\Omega)}\|a\|_{\ell_p} = Cm^{-d/p}.
    \end{equation}
    Here the obvious modifications need to be made if $q = \infty$.
    
    Composing the map $c$ from Proposition \ref{sphere-embedding-lq-lower-bounds-prop} with $\Psi$ and rescaling by $Cm^{-d/q+s}$ gives a map $\tilde{c}:S^d\rightarrow \mathcal{B}_{1,q}^s$ by \eqref{bespv-space-upper-bound}. Moreover, by \eqref{l-p-sum-lower-bound} we have
    \begin{equation}
        \tilde{c}(z) \geq Cm^{d/q - d/p - s}(M-n)^{1/p - 1/q}.
    \end{equation}
    for all $z\in S^d$. Using that $m^d = M$ and $2n \leq M\leq Cn$ finally gives
    \begin{equation}
        \tilde{c}(z) \geq Cn^{-s}
    \end{equation}
    for all $z\in S^d$ as desired.
\end{proof}

\section{Acknowledgements}
We would like to thank Ronald DeVore, Albert Cohen, Guergana Petrova, and Przemyslaw Wojtaszczyk for helpful discussions and Nikola Kovachki and Samuel Lanthaler for bringing this problem to our attention. We would also like to thank the anonymous referees for their helpful comments and suggestions, and Manfred Faldum for pointing out a minor error in a previous version of the manuscript. This work was supported by the National Science Foundation (DMS-2424305 and CCF-2205004) as well as the Office of Naval Research (MURI ONR grant N00014-20-1-2787).

\bibliographystyle{amsplain}
\bibliography{refs}

\providecommand{\bysame}{\leavevmode\hbox to3em{\hrulefill}\thinspace}
\providecommand{\MR}{\relax\ifhmode\unskip\space\fi MR }
\providecommand{\MRhref}[2]{%
  \href{http://www.ams.org/mathscinet-getitem?mr=#1}{#2}
}
\providecommand{\href}[2]{#2}
\begin{thebibliography}{10}

\bibitem{aleksandrov1998combinatorial}
Pavel~S Aleksandrov, \emph{Combinatorial topology}, vol.~1, Courier
  Corporation, 1998.

\bibitem{borsuk1933drei}
Karol Borsuk, \emph{Drei {S}{\"a}tze {\"u}ber die n-dimensionale euklidische
  {S}ph{\"a}re}, Fundamenta Mathematicae \textbf{20} (1933), no.~1, 177--190.

\bibitem{cohen2022optimal}
Albert Cohen, Ronald DeVore, Guergana Petrova, and Przemyslaw Wojtaszczyk,
  \emph{Optimal stable nonlinear approximation}, Foundations of Computational
  Mathematics \textbf{22} (2022), no.~3, 607--648.

\bibitem{devore1989optimal}
Ronald~A DeVore, Ralph Howard, and Charles Micchelli, \emph{Optimal nonlinear
  approximation}, Manuscripta mathematica \textbf{63} (1989), 469--478.

\bibitem{devore1993wavelet}
Ronald~A DeVore, George Kyriazis, Dany Leviatan, and Vladimir~M Tikhomirov,
  \emph{Wavelet compression and nonlinear n-widths.}, Adv. Comput. Math.
  \textbf{1} (1993), no.~2, 197--214.

\bibitem{devore1993constructive}
Ronald~A DeVore and George~G Lorentz, \emph{Constructive approximation}, vol.
  303, Springer Science \& Business Media, 1993.

\bibitem{devore1984maximal}
Ronald~A DeVore and Robert~C Sharpley, \emph{Maximal functions measuring
  smoothness}, vol. 293, American Mathematical Soc., 1984.

\bibitem{devore1993besov}
\bysame, \emph{Besov spaces on domains in $\mathbb{R}^d$}, Transactions of the
  American Mathematical Society \textbf{335} (1993), no.~2, 843--864.

\bibitem{di2012hitchhikerʼs}
Eleonora Di~Nezza, Giampiero Palatucci, and Enrico Valdinoci,
  \emph{Hitchhiker's guide to the fractional {S}obolev spaces}, Bulletin des
  sciences math{\'e}matiques \textbf{136} (2012), no.~5, 521--573.

\bibitem{evans2022partial}
Lawrence~C Evans, \emph{Partial differential equations}, vol.~19, American
  Mathematical Society, 2022.

\bibitem{feng2009bernstein}
Guo Feng and Gen~Sun Fang, \emph{Bernstein n-widths for classes of convolution
  functions with kernels satisfying certain oscillation properties}, Acta
  Mathematica Sinica, English Series \textbf{25} (2009), no.~3, 393--402.

\bibitem{khodulev1989remark}
Andrey~Borisovich Khodulev, \emph{A remark on the {A}leksandrov diameters of
  finite-dimensional sets}, Functional Analysis and Its Applications
  \textbf{23} (1989), no.~2, 165--167.

\bibitem{kolmogoroff1936uber}
A~Kolmogoroff, \emph{{\"U}ber die beste {A}nn\"aherung von {F}unktionen einer
  gegebenen {F}unktionenklasse}, Annals of Mathematics (1936), 107--110.

\bibitem{lecun2015deep}
Yann LeCun, Yoshua Bengio, and Geoffrey Hinton, \emph{Deep learning}, Nature
  \textbf{521} (2015), no.~7553, 436--444.

\bibitem{li2013bernstein}
Yue~Wu Li and Gen~Sun Fang, \emph{Bernstein n-widths of {B}esov embeddings on
  {L}ipschitz domains}, Acta Mathematica Sinica, English Series \textbf{29}
  (2013), no.~12, 2283--2294.

\bibitem{lorentz1996constructive}
George~G Lorentz, Manfred von Golitschek, and Yuly Makovoz, \emph{Constructive
  approximation: advanced problems}, vol. 304, Citeseer, 1996.

\bibitem{petrova2023lipschitz}
Guergana Petrova and Przemys{\l}aw Wojtaszczyk, \emph{Lipschitz widths},
  Constructive Approximation \textbf{57} (2023), no.~2, 759--805.

\bibitem{pinkus2012n}
Allan Pinkus, \emph{N-widths in approximation theory}, vol.~7, Springer Science
  \& Business Media, 2012.

\bibitem{stesin1975aleksandrov}
Mikhail~Isaakovich Stesin, \emph{Aleksandrov diameters of finite-dimensional
  sets and classes of smooth functions}, Doklady Akademii Nauk, vol. 220,
  Russian Academy of Sciences, 1975, pp.~1278--1281.

\bibitem{tikhomirov2012analysis}
Vladimir~M Tikhomirov, \emph{Analysis ii: convex analysis and approximation
  theory}, vol.~14, Springer Science \& Business Media, 2012.

\bibitem{tikhomirov1971some}
Vladimir~Mikhailovich Tikhomirov, \emph{Some problems in approximation theory},
  Mathematical Notes of the Academy of Sciences of the USSR \textbf{9} (1971),
  343--350.

\bibitem{triebel2008function}
Hans Triebel, \emph{Function spaces and wavelets on domains}, no.~7, European
  Mathematical Society, 2008.

\end{thebibliography}
\end{document}